\documentclass{amsart}
\usepackage{amscd}
\usepackage{mathrsfs}
\usepackage{cases}
\usepackage{latexsym}
\usepackage{amsmath}
\usepackage{indentfirst}
\usepackage[arrow,matrix]{xy}
\usepackage{stmaryrd}
\usepackage{amsfonts, color}
\usepackage{amsmath,amssymb,amscd,bbm,amsthm,mathrsfs,dsfont}
\usepackage{graphicx}
\usepackage{fancyhdr}
\usepackage{amsxtra,ifthen}
\usepackage{verbatim}
\usepackage{latexsym,epsfig,amsmath}
\usepackage{epsfig}

\DeclareMathOperator{\End}{End} 
 \DeclareMathOperator{\Ker}{Ker}

 \DeclareMathOperator{\Ann}{Ann}
 
\DeclareMathOperator{\Hom}{Hom}
\DeclareMathOperator{\rad}{rad}
\DeclareMathOperator{\Ext}{Ext}
\DeclareMathOperator{\Img}{Im}

\numberwithin{equation}{section}

\theoremstyle{plain}
\newtheorem{theorem}{Theorem}[section]
\newtheorem{lemma}[theorem]{Lemma}
\newtheorem{proposition}[theorem]{Proposition}
\newtheorem{corollary}[theorem]{Corollary}

\theoremstyle{definition}

\newtheorem{remark}[theorem]{Remark}

\begin{document}

\title[zad modules for associative algebras]
{Zero action determined modules for associative algebras}

\author{Wei Hu, Zhankui Xiao}

\address{Wei Hu, School of Mathematical Sciences, Beijing Normal University, Beijing 100875, P.R. China}
\email{huwei@bnu.edu.cn}

\address{Zhankui Xiao: School of Mathematical Sciences, Huaqiao University,
Quanzhou, Fujian 362021, P.R. China}

\email{zhkxiao@hqu.edu.cn}

\thanks{Both authors are partially supported by the National Natural Science Foundation of China
(Grant No. 11301195, 11471038). Xiao thanks a research foundation of Huaqiao University (Project 2014KJTD14)
and Hu is grateful to the Fundamental Research Funds for the Central Universities for partial support}

\subjclass[2010]{16D70, 15A86, 16P10}
\keywords{Zero action determined module; Zero product determined algebra;
Irreducible module; Principal projective module}

\begin{abstract}
Let $A$ be a unital associative algebra over a field $F$ and $V$ be a unital left $A$-module. The module
$V$ is called zero action determined if every bilinear map $f: A\times V\rightarrow F$ with the property that
$f(a,m)=0$ whenever $am=0$ is of the form $f(x,v)=\Phi(xv)$ for some linear map $\Phi: V\rightarrow F$.
In this paper, we classify the finite dimensional irreducible and principal projective zero action determined modules of $A$.
As an application, two classes of zero product determined algebras are shown:
some semiperfect algebras (infinite dimensional in general);
quasi-hereditary cellular algebras.
\end{abstract}

\maketitle

\section{Introduction}\label{xxsec1}

Let $F$ be a field and $A$ be a unital associative $F$-algebra. We say that an $A$-module $V$ is {\em zero action determined}
({\em zad} for short) if for every $F$-bilinear map $f: A\times V\rightarrow F$ with the property that
\begin{equation}\label{eq1.1}
f(a,m)=0\quad {\rm whenever}\quad am=0,
\end{equation}
then there exists a linear map $\Phi: V\rightarrow F$ such that
\begin{equation}\label{eq1.2}
f(x,v)=\Phi(xv)\quad {\rm for\ all}\quad x\in A, v\in V.
\end{equation}
To the best of our knowledge, this concept was first introduced in \cite{BGLLZ}, studied with respect to Lie algebras.

The notation of a zad module is an extension of the
zero product determined algebra (see Section \ref{xxsec3} for definition), which
was initially introduced in \cite{BGS}. The original motivation for studying zero product determined algebras
is to characterise the zero product preserving linear maps, see \cite{BrSe,CKL,CKLW} for some
historic background. After the seminal paper \cite{BGS}, the concept of zero product determined algebra
was extensively studied for associative and non-associative algebras (see \cite{BH,Grasic10,Grasic11,Grasic15,WangYC}
and the reference therein). At the same time, more applications of zero product determined algebras
were found \cite{BG,WangYC}, such as commutativity preserving linear maps, maps derivable at zero, etc.
On the other hand, the analogues of zero product determined algebras in functional analysis
have been attracting much more attention \cite{ABESV,ABEV}, where one additionally assumes that the maps in
question are continuous. We warmly encourage the reader to Bre\v{s}ar's paper \cite{Bresar} for
a comprehensive understanding of this topic.

Similar with \cite{BGLLZ}, we originally wanted to make the zad modules as a tool
to study the zero product determined algebras, more explicitly, to find some new classes of
zero product determined algebras, especially in the infinite dimensional case. It seems natural
to study the annihilator ideal of the $A$-module $V$ if we need to build a bridge between
the zero product determined property of $A$ and the zad property of $V$. When we studied the
annihilator of indecomposable direct summands of the regular $A$-module, the results on our draft turned out
to be a part of the blueprint of classifying all the (finite dimensional) zad $A$-modules.
We then accordingly changed our target, which led to this paper as our first attempt to
study the structure and representation theory of $A$ through the zad $A$-modules.

The content of this article is organized as follows. Section \ref{xxsec2} gives some general
observations for zad modules, for most of which one can find analogues in \cite[Section 2]{Bresar}.
In Section \ref{xxsec3}, we characterize the finite dimensional irreducible zad $A$-modules and
principal projective zad $A$-modules. As a consequence of the characterization, some
semiperfect algebras (infinite dimensional in general) and
quasi-hereditary cellular algebras are shown to be zero product determined.
Let us end the introduction with some terminological convention. Throughout the paper,
all {\em algebras} are {\em unital associative algebras} over a fixed field $F$, and all
{\em modules} are {\em unital left modules}. Both algebras and modules may be
infinite dimensional unless separate declaration.

\section{Basic Properties}\label{xxsec2}

In this section we gather together several basic properties of zad modules. One may find many of
them are analogues of the results in \cite[Section 2]{Bresar}, but we exhibit here in a self-contained
manner for completeness. Let $A$ be
an $F$-algebra and $V$ be an $A$-module. We start by remarking that (\ref{eq1.2}) is equivalent to
\begin{equation}\label{eq2.1}
f(xy,v)=f(x,yv)\quad {\rm for\ all}\quad x,y\in A, v\in V,
\end{equation}
as well as to
\begin{equation}\label{eq2.2}
f(x,v)=f(1,xv)\quad {\rm for\ all}\quad x\in A, v\in V.
\end{equation}
In the following we will use (\ref{eq2.1}) or (\ref{eq2.2}) instead of (\ref{eq1.2}) without comments.

The following lemma shows that the bilinear map from the definition of a zad module may has its range
in an arbitrary $F$-linear space, not necessarily in the field $F$ itself. This elementary fact is
important for applications.

\begin{lemma}\label{xx2.1}
Let $V$ be an $A$-module. Then $V$ is zad if and only if for every bilinear map $f$ from $A\times V$
into an arbitrary vector space $U$ with the property that
$$
f(a,m)=0\quad{\text whenever}\quad am=0,
$$
then there exists a linear map $\Phi: V\rightarrow U$ such that
$$
f(x,v)=\Phi(xv)\quad {\text for\ all}\quad x\in A, v\in V.
$$
\end{lemma}

\begin{proof}
We only need to prove the necessity. For every bilinear map $f: A\times V\rightarrow U$ satisfying
(\ref{eq1.1}) with $U$ an arbitrary $F$-space, we take an arbitrary linear functional $\alpha$ on $U$
and define $f_{\alpha}:=\alpha \circ f: A\times V\rightarrow F$. It is clear that $f_{\alpha}$ satisfies
(\ref{eq1.1}) and hence $f_{\alpha}(x,v)=f_{\alpha}(1,xv)$ for all $x\in A, v\in V$. Since $\alpha$
is an arbitrary linear functional on $U$, this actually implies that $f$ satisfies (\ref{eq2.2}).
\end{proof}

\begin{lemma}\label{xx2.2}
The $A$-module $V$ is zad if and only if for every bilinear map $f:A\times V\rightarrow F$ satisfying
{\rm (\ref{eq1.1})}, there is $\sum_i f(x_i,v_i)=0$ whenever $\sum_i x_iv_i=0$.
\end{lemma}

\begin{proof}
We only need to prove the sufficiency. If $f:A\times V\rightarrow F$ is a bilinear map such that
$\sum_i f(x_i,v_i)=0$ whenever $\sum_i x_iv_i=0$, then the linear map $\Phi: V\rightarrow F$ given by
$$
\Phi\left(\sum_i x_iv_i\right):=\sum_i f(x_i,v_i)
$$
is well-defined and satisfies (\ref{eq1.2}).
\end{proof}

By the above lemma, it is reasonable to study the kernel $T_{ker}$ of the linear map $A\otimes V\rightarrow V$
defined by $x\otimes v \mapsto xv$, i.e.,
$$
T_{ker}=\left\{\sum_i x_i\otimes v_i\in  A\otimes V\ \left| \  \sum_i x_iv_i=0\right\}\right..
$$
Let us define
$$
S_{A\otimes V}:=\{x\otimes v\in A\otimes V\ |\ xv=0\}.
$$
It is clear that $F$-span $S_{A\otimes V}\subseteq T_{ker}$.

\begin{lemma}\label{xx2.3}
The $A$-module $V$ is zad if and only if $F{\rm{\text-span}}\ S_{A\otimes V}=T_{ker}$.
\end{lemma}

\begin{proof}
Suppose that $V$ is zad and $F{\rm{\text-span}}\ S_{A\otimes V}\subsetneq T_{ker}$. Then there exists a linear functional
$\alpha$ on $A\otimes V$ such that $\alpha(S_{A\otimes V})=0$ but $\alpha(T_{ker})\neq 0$. The bilinear
map $f: A\times V\rightarrow F$ defined by $f(x,v)=\alpha(x\otimes v)$ satisfies (\ref{eq1.1}) but not
(\ref{eq1.2}). To prove the converse, by the universal property of tensor product,
for every bilinear map $f: A\times V\rightarrow F$, there exists a unique linear
map $\alpha_f: A\otimes V\rightarrow F$ such that $\alpha_f(x\otimes v)=f(x,v)$. Now assume $f$ satisfies
the condition (\ref{eq1.1}). Let $\sum_i x_i\otimes v_i\in  A\otimes V$ with $\sum_i x_iv_i=0$.
Then there exist some $a_j\otimes m_j\in  S_{A\otimes V}$ such that $\sum_j a_j\otimes m_j=
\sum_i x_i\otimes v_i$. Hence
$$\begin{aligned}
\sum_i f(x_i,v_i)&=\sum_i \alpha_f(x_i\otimes v_i)=\alpha_f \bigr(\sum_i x_i\otimes v_i\bigr)\\
&=\sum_j \alpha_f(a_j\otimes m_j)=\sum_j f(a_j,m_j)=0.
\end{aligned}$$
Then $V$ is zad follows from Lemma \ref{xx2.2}.
\end{proof}

In practice, we find that it is difficult to determine whether the identity $F$-span $S_{A\otimes V}=T_{ker}$
holds or not, mainly because one element in $A\otimes V$ may have different expressions. The following
lemma provides an, to some extend, easy way to study the zad property of $A$-modules.

\begin{lemma}\label{xx2.4}
The $A$-module $V$ is zad if and only if
$$
x\otimes v-1\otimes xv\in F{\rm{\text-span}}\ S_{A\otimes V}
$$
for every $x\in A, v\in V$.
\end{lemma}

\begin{proof}
If $V$ is a zad $A$-module, it follows from Lemma \ref{xx2.3} that
$x\otimes v-1\otimes xv\in T_{ker}=F{\rm{\text-span}}\ S_{A\otimes V}$.
Conversely, for an arbitrary bilinear map $f: A\times V\rightarrow F$ satisfying (\ref{eq1.1}),
by the universal property of tensor product, there exists a unique linear
map $\alpha_f: A\otimes V\rightarrow F$ such that $\alpha_f(x\otimes v)=f(x,v)$.
Then $\alpha_f$ vanishes on the subspace $F$-span $S_{A\otimes V}$. Hence
$\alpha_f(x\otimes v-1\otimes xv)=0$ for every $x\in A, v\in V$. In other words, $f$ satisfies (\ref{eq2.2})
and thus $V$ is a zad $A$-module.
\end{proof}

\begin{lemma}\label{xx2.5}
A homomorphic image of a zad $A$-module is also zad.
\end{lemma}

\begin{proof}
Let $\varphi: V\rightarrow U$ be a surjective homomorphism of $A$-modules. If $V$ is zad, then
$x\otimes v-1\otimes xv\in F{\rm{\text-span}}\ S_{A\otimes V}$ for all $x\in A, v\in V$ by Lemma \ref{xx2.4}.
For every $a\in A, u\in U$, the surjection of $\varphi$ shows that $a\otimes u-1\otimes au\in
F{\rm{\text -span}}\{x\otimes \varphi(v)\in A\otimes U\ |\ xv=0\}$. Note that
$F{\rm{\text -span}}\{x\otimes \varphi(v)\in A\otimes U\ |\ xv=0\}$ is a subset of $F$-span $S_{A\otimes U}$,
since $x\varphi(v)=\varphi(xv)=0$. Then again using Lemma \ref{xx2.4} we obtain that $U$ is a zad $A$-module.
\end{proof}

The following result is elementary but is needed.

\begin{lemma}\label{xx2.6}
The $A$-modules $V_1, V_2,\ldots, V_n$ are zad if and only if their direct sum $\oplus_{i=1}^n V_i$ is zad.
\end{lemma}

\begin{proof}
The ``if" part follows from Lemma \ref{xx2.5}. For the ``only if" part it is enough to consider the case of $n=2$.
Thus let $V_1$ and $V_2$ be two zad $A$-modules and let $U:=V_1\oplus V_2$. For every bilinear map $f:A\times U
\rightarrow F$ satisfying (\ref{eq1.1}), we have $f(a,(m_1,0))=0$ if $am_1=0$, where $a\in A, m_1\in V_1$.
The zad property of $V_1$ implies that
$$
f(x,(v_1,0))=f(1,(xv_1,0))
$$
for all $x\in A, v_1\in V_1$. Similarly,
$$
f(x,(0,v_2))=f(1,(0,xv_2))
$$
for all $x\in A, v_2\in V_2$. Therefore we have
$$\begin{aligned}
f(x,(v_1,v_2))&=f(x,(v_1,0))+f(x,(0,v_2))=f(1,(xv_1,0))+f(1,(0,xv_2))\\
&=f(1,(xv_1,xv_2))=f(1,x(v_1,v_2))
\end{aligned}$$
for all $x\in A$ and $v_1\in V_1, v_2\in V_2$. Hence $f$ satisfies (\ref{eq2.2}) and $U$ is zad.
\end{proof}

We now end this section by an application of the zad $A$-modules. Let $V$ and $U$ be two $A$-modules.
A linear map $\varphi: V\rightarrow U$ is called a {\em zero action preserving} map, if for every
$a\in A, m\in V$, there is $a\varphi(m)=0$ whenever $am=0$.

\begin{proposition}\label{xx2.7}
Let $V$ and $U$ be two $A$-modules. Let $\varphi: V\rightarrow U$ be a zero action preserving linear map.
If $V$ is zad, then $\varphi$ is a homomorphism of $A$-modules.
\end{proposition}

\begin{proof}
Define a bilinear map $f:A\times V\rightarrow U$ by $f(x,v)=x\varphi(v)$. Then Lemma \ref{xx2.1}
implies that there exists a linear map $\Phi: V\rightarrow U$ such that $f(x,v)=\Phi(xv)$ for all
$x\in A, v\in V$. Note that $A$ is unital. We have $\Phi=\varphi$.
\end{proof}

The problem becomes interesting if we do not assume the algebra $A$ is unital. Furthermore, it is hopeful that
the basic idea in Proposition \ref{xx2.7} will play an important role in the functional analytic context,
see \cite[Subsection 2.4]{Bresar}. However, we do not plan to study here on this perspective
because of the main target of this paper mentioned in the introduction.

\section{Finite Dimensional Zero Action Determined Modules}\label{xxsec3}

In this section, we study the structure and representation theory of $A$ through the zad $A$-modules.
All the terminologies and basic results related to representation theory, which we used, one can find in the book \cite{CuRe}.
Let us start by recalling the definition of a {\em zero product determined} algebra.

An $F$-algebra $A$ is zero product determined if for every bilinear map $f: A\times A\rightarrow F$
with the property that
$$
f(a,b)=0\quad {\rm whenever}\quad ab=0,
$$
then there exists a linear map $\Phi: A\rightarrow F$ such that
$$
f(x,y)=\Phi(xy)\quad {\rm for\ all}\quad x,y\in A.
$$
The following elementary observation is need.

\begin{lemma}\label{xx3.1}
Let $A$ be an algebra. The following statements are equivalent:
\begin{enumerate}
\item[(1)] the algebra $A$ is zero product determined;

\item[(2)] every $A$-module $V$ is zad;

\item[(3)] the regular $A$-module is zad.
\end{enumerate}
\end{lemma}

\begin{proof}
$(1)\Rightarrow (2)$. It follows from \cite[Remark 2.11]{BGLLZ}. We give the proof here for the reader's convenience.
For every bilinear map $f: A\times V\rightarrow F$ satisfying (\ref{eq1.1}),
and for every $v\in V$, we define a bilinear map $f_v: A\times A\rightarrow F$ by
$f_v(x,y)=f(x,yv)$. It is clear that $f_v(a,b)=0$ whenever $ab=0$. Since $A$ is zero product determined,
$f_v(x,1)=f_v(1,x)$ for all $x\in A$. In other words, $f(x,v)=f(1,xv)$ for all $x\in A, v\in V$.
We have $f$ satisfies (\ref{eq2.2}) and hence $V$ is zad.

The claims $(2)\Rightarrow (3)$ and $(3)\Rightarrow (1)$ are clear.
\end{proof}

Let $V$ be an $A$-module. We denote $\Ann_{A}(V)$ the annihilator ideal of $V$ in $A$, i.e.,
$$
\Ann_{A}(V)=\{a\in A\ |\ av=0,\ \forall\ v\in V\}.
$$
Let $J$ be an ideal of $A$ contained in $\Ann_A(V)$. Then $V$ becomes an $A/J$-module in a natural way.
The following useful lemma will be frequently used in later proofs.

\begin{lemma}\label{lemma-zad-over-quotient}
Let $A$ be an algebra and $V$ be an $A$-module. Suppose that $J$ is an ideal of $A$ contained in $\Ann_A(V)$.
Then $V$ is a zad $A$-module if and only if $V$ is a zad $A/J$-module.
\end{lemma}

\begin{proof}
For simplicity, we write $\bar{A}$ for $A/J$.

$(\Rightarrow)$. For every bilinear map $f: \bar{A}\times V\rightarrow F$ satisfying (\ref{eq1.1}),
we define a bilinear map $\hat{f}: A\times V\rightarrow F$ by $\hat{f}(x,v)=f(\bar{x},v)$, where
$\bar{x}$ denote the canonical image of $x$ in $\bar{A}$. Clearly, $\hat{f}$ satisfies (\ref{eq1.1})
and hence $\hat{f}(x,v)=\hat{f}(1,xv)$ for all $x\in A, v\in V$. Therefore $f(\bar{x},v)=f(\bar{1},xv)
=f(\bar{1},\bar{x}v)$, i.e., $f$ satisfies (\ref{eq2.2}).

$(\Leftarrow)$. For every bilinear map $f: A\times V\rightarrow F$ satisfying (\ref{eq1.1}),
we define $\bar{f}: \bar{A}\times V\rightarrow F$ by $\bar{f}(\bar{x},v)=f(x,v)$. We claim that
$\bar{f}$ is well-defined. In fact, for any $x_1,x_2\in A$, if $\bar{x}_1=\bar{x}_2$, we have
$x_1-x_2\in J\subseteq \Ann_{A}(V)$. Then $(x_1-x_2)v=0$ for all $v\in V$, and hence $f(x_1-x_2,v)=0$.
In other words, $f(x_1,v)=f(x_2,v)$ for all $v\in V$ and $\bar{f}$ is well-defined.
Now it is easy to check that $\bar{f}$ is a bilinear map satisfying (\ref{eq1.1}).
Since $V$ is a zad $\bar{A}$-module, $\bar{f}(\bar{x},v)=\bar{f}(\bar{1},\bar{x}v)=\bar{f}(\bar{1},xv)$
for all $x\in A, v\in V$. Therefore, $f$ satisfies (\ref{eq2.2}).
\end{proof}

We are now in a position to give our first main result, which classifies finite dimensional
irreducible zad $A$-modules.

\begin{theorem}\label{xx3.3}
Let $V$ be a finite dimensional irreducible $A$-module. Then $V$ is zad if and only if
$\dim V> \dim \End_A(V)$ or $\dim \End_A(V)=1$.
\end{theorem}

\begin{proof}
Since $V$ is finite dimensional, then $\bar{A}:=A/ \Ann_{A}(V)$ is also finite dimensional,
which has a faithful irreducible module. Therefore, $\bar{A}$ is a simple algebra.
The well-known Wedderburn-Artin theorem implies that $\bar{A}$ is isomorphic to the full
matrix algebra $M_n(D)$, with $D$ a division $F$-algebra and $n\geq 1$.
Up to isomorphism, $V$ is the unique irreducible $\bar{A}$-module and the regular
$\bar{A}$-module is isomorphic to $V^{\oplus n}$. Particularly $\dim V=n\dim D$.

$(\Rightarrow)$. If $V$ is a zad $A$-module, then $V$ is a zad $\bar{A}$-module by Lemma \ref{lemma-zad-over-quotient}.
From Lemma \ref{xx2.6} and Lemma \ref{xx3.1} we
deduce that the algebra $\bar{A}$ is zero product determined. Assume that $\dim \End_A(V)\neq1$.
Then $\dim D>1$ since $\End_A(V)=\End_{\bar{A}}(V)\cong D^{\rm op}$, the opposite algebra of $D$.
We have known that $M_n(D)$ is zero product determined for $n\geq 2$ by \cite[Proposition 2.18]{Bresar},
see also \cite{BGS}, and $D$ is zero product determined if and only if $D=F$ by \cite[Proposition 2.8]{Bresar}.
Therefore, $n\geq 2$ and $\dim V=n\dim D> \dim \End_A(V)$

$(\Leftarrow)$. If $\dim \End_A(V)=1$, then $D=F$ since $\End_A(V)\cong D^{\rm op}$, which in turn
implies that $\bar{A}\cong M_n(F)$ is a zero product determined algebra by \cite[Propositions 2.8 and 2.18]{Bresar}.
If $\dim V> \dim \End_A(V)$, then  $n\geq 2$, and  $\bar{A}\cong M_n(D)$ is a zero product determined algebra by \cite[Proposition 2.18]{Bresar}.
Hence, in both cases, $V$ is a zad $A$-module by Lemma \ref{xx3.1} and Lemma \ref{lemma-zad-over-quotient}.
This completes the proof of the theorem.
\end{proof}

For a finite dimensional irreducible $A$-module $V$, to check the zad property of $V$, Theorem \ref{xx3.3}
tells us that we only need to compute the dimension of the invariants $\End_A(V)$.
If $V$ is an infinite dimensional irreducible $A$-module, then $\bar{A}$ is primitive.
The proof of Theorem \ref{xx3.3} exhibits implicitly that to find more examples of infinite dimensional
zero product determined algebras, it is better to work in the primitive algebras over
an algebraically closed field.

Let $A$ be a finite dimensional $F$-algebra. An finite dimensional irreducible $A$-module
$V$ is said to be absolutely irreducible, if $K\otimes V$ is an irreducible $(K\otimes A)$-module
for any field extension $F\subseteq K$, or equivalently, $\End_A(V)\cong F$. A finite dimensional algebra $A$ is splitting (or, the ground field $F$ is a splitting field for $A$ ) if all irreducible $A$-modules are absolutely irreducible.
The following two corollaries of Theorem \ref{xx3.3}
are of special interest for us.

\begin{corollary}\label{xx3.4}
Let $A$ be a finite dimensional $F$-algebra. If $V$ is an absolutely irreducible $A$-module,
then $V$ is zad.
\end{corollary}

\begin{corollary}\label{xx3.5}
Let $A$ be a finite dimensional splitting $F$-algebra.
then every irreducible $A$-module is zad.
\end{corollary}

There is a large class of finite dimensional algebras, called cellular algebras
\cite{GL}, for which the ground fields are always splitting. Many important algebras from representation
theory are in fact cellular, such as, Hecke algebras of finite type, $q$-Schur algebras, Ariki-Koike algebras,
Brauer algebras, Birman-Wenzl algebras, etc. see \cite{AST} and the references therein for more examples.
We believe that, in many cases, these algebras are zero product determined (particularly, quasi-hereditary
cellular algebras are zero product determined, see Corollary \ref{xx3.10} below). Then there should
exist presentations of these algebras in which all generators are idempotents by \cite[Theorem 3.7]{Bresar},
which in turn can exhibit more information of cellular algebras.

\medskip
We then turn to classify the principal (indecomposable) projective zad modules for a finite dimensional algebra.
The main result, Theorem \ref{thm-zad-proj}, implies that a finite dimensional splitting algebra with
finite global dimension is zero product determined. Note that all quasi-hereditary cellular algebras
have finite global dimensions, and are therefore zero product determined.
This provides a lot of interesting examples of zero product determined algebras.

\begin{theorem}\label{thm-zad-proj}
Suppose that $A$ is a finite dimensional algebra. Let $E$ be the subalgebra of $A$ generated by all idempotents,
and let $I$ be the ideal of $A$ generated by all commutators of idempotents with arbitrary elements in $A$.
The Jacobson radical of $A$ is denoted by $R$. Let $e$ be an idempotent in $A$. Then the following statements are equivalent:
\begin{enumerate}
\item[(1)] the principal projective $A$-module $Ae$ is zad;

\item[(2)] $Ae/Re$ is zad and $Re\subseteq Ie$;

\item[(3)] $Ae/Re$ is zad and every $1$-dimensional quotient module $S$ of $Ae$ satisfies $\Ext_A^1(S,S)=0$;

\item[(4)]$Ae=Ee$.
\end{enumerate}
\end{theorem}

What is surprising here is that, in the statement (3) of Theorem \ref{thm-zad-proj}, only irreducible modules are involved.
Before giving the proof, we need several lemmas. The first lemma is a generalization of \cite[Proposition 2.8]{Bresar}.

\begin{lemma}\label{lemma-local-zpd-algebra}
Let $B$ be a zero product determined $F$-algebra (maybe infinite dimensional). If $B$ is local,
then $B=F$.
\end{lemma}

\begin{proof}
Let $\mathfrak{m}$ be the unique maximal ideal of $B$, which is the Jacobson radical of $B$.
Then $B/\mathfrak{m}$ is a zero product determined $F$-algebra by \cite[Proposition 2.4]{Bresar}.
Note that $B/\mathfrak{m}$ is a division algebra and hence $B/\mathfrak{m}\cong F$ by
\cite[Proposition 2.8]{Bresar}. Now \cite[Proposition 2.10]{Bresar} shows $\mathfrak{m}=0$.
\end{proof}

Recall that a module is said to be cyclic provided it is generated by one element.

\begin{lemma}\label{lemma-local-zad-module}
Let $B$ be a local $F$-algebra (maybe infinite dimensional). Let $X$ be a nonzero finite dimensional cyclic $B$-module.
Then $X$ is zad as a $B$-module if and only if $X$ is $1$-dimensional.
\end{lemma}

\begin{proof}
We only need to prove the necessity by Theorem \ref{xx3.3}.
Let $\mathfrak{m}$ be the unique maximal ideal of $B$.
We denote $D:=B/\mathfrak{m}$. Then $D$ is a division algebra. We first prove that $D\cong F$.
Actually, by Nakayama's Lemma, the quotient $X/\mathfrak{m}X$ is a nonzero cyclic $D$-module.
This implies that $X/\mathfrak{m}X\cong D$ as $D$-modules. By Lemmas \ref{xx2.5} and \ref{lemma-zad-over-quotient},
the quotient $X/\mathfrak{m}X$ is a zad $D$-module. This implies that $D$ is a zero product determined algebra.
It follows from Lemma \ref{lemma-local-zpd-algebra} (or \cite[Proposition 2.8]{Bresar}) that $D\cong F$.

To show $X$ is $1$-dimensional, it suffices to show that $\mathfrak{m}X=0$.
Suppose contrarily that $\mathfrak{m}X\neq 0$. Then, by Nakayama's Lemma, $\mathfrak{m}^2X\neq \mathfrak{m}X$,
so $X/\mathfrak{m}^2X$ is not isomorphic to $X/\mathfrak{m}X$ and has dimension $>1$.
Since $X$ is a cyclic $B$-module, $X/\mathfrak{m}^2X$ is also a cyclic $B$-module.
For simplicity, we denote $X/\mathfrak{m}^2X$ by $\bar{X}$. Then $\mathfrak{m}^2\bar{X}=0$.
Let $\bar{x}$ be a generator of $\bar{X}$. Then
there is a surjective $B$-module homomorphism $\pi: B\rightarrow \bar{X}$ sending $b$ to $b\bar{x}$ with kernel $\Ann_B(\bar{x})$.
We claim that $\Ann_B(\bar{X})= \Ann_B(\bar{x})$. The inclusions  $\Ann_B(\bar{X})\subseteq \Ann_B(\bar{x})\subseteq \mathfrak{m}$
are obvious. Conversely, let $m_1\in\Ann_B(\bar{x})$. Since $B/\mathfrak{m}\cong F$, every element in $B$ is of the form
$\lambda+m$ with $\lambda\in F$ and $m\in\mathfrak{m}$. Thus $m_1(\lambda+m)\bar{x}=m_1\lambda\bar{x}+m_1m\bar{x}=\lambda m_1\bar{x}=0$.
Hence $\Ann_B(\bar{x})\subseteq\Ann_B(\bar{X})$ and consequently $\Ann_B(\bar{x})=\Ann_B(\bar{X})$.
It follows that
$$\bar{X}\cong B/\Ker\pi=B/\Ann_B(\bar{x})=B/\Ann_B(\bar{X}).$$
This means that $\bar{X}$ is isomorphic to the regular module of $B/\Ann_B(\bar{X})$.
By Lemmas \ref{xx2.5}, \ref{lemma-zad-over-quotient} and \ref{xx3.1}, we deduce that $B/\Ann_B({\bar{X}})$
is a zero product determined algebra, which is local. Then Lemma \ref{lemma-local-zpd-algebra}
forces that $B/\Ann_B(\bar{X})\cong F$ and $\bar{X}\cong F$, which is a contradiction since $\dim \bar{X}>1$.
\end{proof}

\begin{lemma}\label{lemma-e1e2-img}
Let $B$ be a algebra and $e_1, e_2$ be idempotents in $B$ such that $e_1e_2=0=e_2e_1$.
Then every $B$-module homomorphism $g: Be_1\rightarrow Be_2$ has $\Img g\subseteq Ie_2$,
where $I$ is the ideal generated by all commutators of idempotents in $B$ with arbitrary elements in $B$.
\end{lemma}

\begin{proof}
It is easy to say that $g(e_1)=e_1xe_2$ for some $x\in B$. Since $e_1xe_2=[e_1, xe_2]e_2\in Ie_2$,
the image $g(be_1)=bg(e_1)$ of any element $be_1$ belongs to $Ie_2$.
\end{proof}

Now we are ready to give a proof of Theorem \ref{thm-zad-proj}.

\begin{proof}[Proof of Theorem \ref{thm-zad-proj} ]
Without loss of generality, we can assume that $e$ is primitive. Actually,
$e$ can be written as a sum $e=e_1+\cdots+e_m$ of pairwise orthogonal primitive idempotents,
and it is easy to see that the statements hold for $e$ if and only if they hold for all $e_i$.

From now on, we assume that $e$ is primitive. In this case, the algebra $eAe$ is a local algebra and so is
the quotient algebra $A/A(1-e)A$. Actually, $A/A(1-e)A$ is isomorphic to a quotient algebra of $eAe$.

\medskip
{\it Claim 1}. If $Ae/Ie$ is a nonzero zad $A$-module, then $Ae/Ie\cong F$.

\smallskip
{\it Proof of Claim 1}. Note that each element $(1-e)ae$ in $(1-e)Ae$ is equal to the commutator $[1-e, ae]$.
This implies that $A(1-e)Ae\subseteq Ie$. Hence $A(1-e)A\subseteq \Ann_A(Ae/Ie)$.
By Lemma \ref{lemma-zad-over-quotient}, $Ae/Ie$ is a zad cyclic module over the local algebra $A/A(1-e)A$.
Thus $Ae/Ie$ has dimension $1$ by Lemma \ref{lemma-local-zad-module}.

\medskip
$(1)\Rightarrow (2)$. That $Ae/Re$ and $Ae/Ie$ are zad follows from Lemma \ref{xx2.5}.
It follows from Claim 1 that $Ae/Ie$ is irreducible as an $A$-module,
and hence $Re\subseteq Ie$ since $Re=\rad(Ae)$, the radical of $Ae$.

\medskip
$(2)\Rightarrow (4)$. It is known that $I\subseteq E$ (see \cite[Lemma 3.1]{Bresar}) and consequently
$Ie\subseteq Ee \subseteq Ae$. If $Ie=Ae$, then $Ee=Ae$. Suppose that $Ie\neq Ae$. Then $Ae/Ie\neq 0$.
Since $Ae/Re$ is zad and $Re\subseteq Ie$, it follows that $Ae/Ie$ is a quotient module of $Ae/Re$,
and is again zad as an $A$-module. By Claim 1, $Ae/Ie$ is $1$-dimensional.
Now $e\not\in Ie$ since $Ie\neq Ae$. However $e\in Ee$. This means that $Ee/Ie$
is a nonzero subspace of the $1$-dimensional space $Ae/Ie$  and has to be the whole $Ae/Ie$. Hence $Ee=Ae$.

\medskip

$(4)\Rightarrow (1)$. Let $f: A\times Ae\rightarrow F$ be a bilinear map satisfying (\ref{eq1.1}),
we need to prove that $f(ab, ce)=f(a, bce)$ for all $a, b, c\in A$, i.e., $f$ satisfies (\ref{eq2.1}).
Firstly, this is true when $b$ is an idempotent. Actually, in this case, $b(1-b)=0$, and $f(ab, ce)=f(ab,bce)+f(ab,(1-b)ce)=f(ab,bce)$.
Similarly, $f(a,bce)=f(ab,bce)+f(a(1-b),bce)=f(ab,bce)$. Hence $f(ab,ce)=f(a,bce)$ for all $a,c\in A$.
Since the elements in $E$ are linear combinations of products of idempotents, we have $f(ax, ce)=f(a,xce)$ for all $a,c\in A$ and $x\in E$.
Now by the hypothesis (4), for all $a,b,c\in A$, there are equalities
$$f(ab,ce)=f(ab,cee)=f(abce,e)=f(a,bcee)=f(a,bce)$$
The second and third equalities hold because $ce$ and $bce$ belong to $Ae=Ee\subseteq E$.

\medskip
$(1)\Rightarrow (3)$. That $Ae/Re$ is zad is clear. Note that $S:=Ae/Re$ is the unique irreducible quotient module of $Ae$.
If $S$ is not $1$-dimensional or $\Ext_A^1(S,S)=0$, there is nothing to prove.
Assume that $S$ is $1$-dimensional and $\Ext_A^1(S,S)\neq 0$. This means there is a non-split short exact sequence
$$0\rightarrow S\rightarrow X\rightarrow S\rightarrow 0$$
of $A$-modules. In other words, $X$ is not isomorphic to the direct sum $S\oplus S$.
This implies that $X$ is a $2$-dimensional cyclic $A$-module. Note that $Ae$ is a projective $A$-module,
the natural surjective homomorphism $Ae\rightarrow S$ lifts to a surjective $A$-homomorphism from $Ae$ to $X$ in this case.
Hence $X$ is isomorphic to a quotient module of $Ae$ and is therefore zad as an $A$-module.
Since $S$ is $1$-dimensional, the multiplicity of $Ae$ as a direct summand of the regular $A$-module ${}_AA$ is $1$.
This means there is no direct summands of $A(1-e)$ having $S$ as a quotient. Hence $(1-e)S\cong \Hom_A(A(1-e), S)=0$.
It follows that $(1-e)X=0$. Thus $X$ is a zad module over the local algebra $A/A(1-e)A$ by Lemma \ref{lemma-zad-over-quotient}.
This leads to a contradiction by Lemma \ref{lemma-local-zad-module}.

\medskip
$(3)\Rightarrow (2)$. We only need to prove that $Re\subseteq Ie$. Two cases occur.

\smallskip
Case (i): $Ae$ is isomorphic to a direct summand of $A(1-e)$.
Then there are $A$-module homomorphisms $g: Ae\rightarrow A(1-e)$ and $h: A(1-e)\rightarrow Ae$ such that $hg=1_{Ae}$.
By Lemma \ref{lemma-e1e2-img}, we see that $Ae=\Img hg\subseteq \Img h\subseteq Ie$, which in turn implies that $Re\subseteq Ie$.

Case (ii): $Ae$ is not isomorphic to any direct summand of $A(1-e)$. In this case,
the multiplicity of $Ae$ as a direct summand of ${}_AA$ is $1$.
This means that the simple quotient $Ae/Re$ is $1$-dimensional over its endomorphism algebra.
In particular, $\dim Ae/Re=\dim \End_{A}(Ae/Re)$. Since $Ae/Re$ is zad,
by Theorem \ref{xx3.3}, we have $\End_A(Ae/Re)\cong F$. Thus $Ae/Re$ is also $1$-dimensional,
and is the unique $1$-dimensional quotient module of $Ae$. Denote $Ae/Re$ by $S$.
Applying $\Hom_A(-,S)$ to the exact sequence $0\rightarrow Re\rightarrow Ae\rightarrow S\rightarrow 0$
results in an exact sequence
$$0\rightarrow \Hom_A(S, S)\stackrel{\theta}{\rightarrow} \Hom_A(Ae, S)\rightarrow \Hom_A(Re, S)\rightarrow \Ext_A^1(S, S).$$
In the above exact sequence, the map $\theta$ is an isomorphism. Since $\Ext_A^1(S, S)=0$,
it follows that $\Hom_A(Re, S)=0$, and thus $Re/R^2e$ does not have $S$ as a direct summand.
Hence $Ae$ is not a direct summand of the projective cover of $Re$. This implies that there is a surjective $A$-module homomorphism
$$\bigoplus_{i=1}^n A(1-e)\stackrel{[g_1,\cdots,g_n]}{\longrightarrow} Re$$
for some $n$. Let $\iota: Re\rightarrow Ae$ be the inclusion map.
Then by Lemma \ref{lemma-e1e2-img}, we have $\Img(\iota g_i)\subseteq Ie$ for all $i=1,\cdots, n$.
It follows that $Re=\sum_{i=1}^n\Img (\iota g_i)\subseteq Ie$.
\end{proof}

\begin{remark}
The proof of Theorem \ref{thm-zad-proj} also works under the condition that $A$ is a semiperfect $F$-algebra
with $A/R$ finite dimensional. An algebra $A$ is semiperfect provided that $1$ can be written as
a sum of pairwise orthogonal idempotents $1=e_1+\cdots+e_n$ such that $e_iAe_i$ is a local algebra for all $i=1,\ldots,n$.
A semiperfect algebra can be infinite dimensional in general.
\end{remark}

Taking $e$ to be the identity $1$ in Theorem \ref{thm-zad-proj}, we obtain the following corollary.

\begin{corollary}\label{corollary-zpd-equiv-condition}
Suppose that $A$ is a finite dimensional algebra. Let $E$ be the subalgebra of $A$ generated by all idempotents,
and let $I$ be the ideal of $A$ generated by all commutators of idempotents with arbitrary elements in $A$.
The Jacobson radical of $A$ is denoted by $R$. Then the following statements are equivalent:
\begin{enumerate}
\item[(1)] $A$ is zero product determined;

\item[(2)] $A/R$ is zero product determined and $R\subseteq I$;

\item[(3)] $A/R$ is zero product determined and every $1$-dimensional $A$-module $S$ has $\Ext_A^1(S, S)=0$;

\item[(4)]$A=E$.
\end{enumerate}
\end{corollary}

{\it Remark.} The equivalence between $(1)$ and $(4)$ in Corollary \ref{corollary-zpd-equiv-condition} was proved in \cite{Bresar}.
The equivalent statement $(2)$ is contained implicatively in the proof of \cite[Theorem 3.7]{Bresar}.  The condition (3) is new and turns out to be a convenient characterization of zero product determined algebras, reducing the problem to irreducible modules.

\medskip
Recall that the global dimension of an algebra $A$ is the supremum of projective dimensions of $A$-module.

\begin{corollary}\label{xx3.10}
Let $A$ be a finite dimensional splitting $F$-algebra.
If $A$ has finite global dimension, then $A$ is zero product determined.
Particularly, all quasi-hereditary cellular algebras over a field are zero product determined.
\end{corollary}

\begin{proof}
Let $R$ be the Jacobson radical of $A$. Since $A$ is splitting, the quotient algebra $A/R$
is a direct sum of full matrix algebras over $F$, which are zero product determined.
Hence $A/R$ is zero product determined. If $A$ has finite global dimension,
then it follows from \cite[Corollary 5.6]{Igusa} that every $1$-dimensional $A$-module $S$ satisfies $\Ext_A^1(S, S)=0$.
By Corollary \ref{corollary-zpd-equiv-condition} (3), we deduce that $A$ is  zero product determined.
\end{proof}

%{\bf Acknowledgements.} The authors express their sincere gratitude to the anonymous
%referees for their substantial and insightful comments which significantly help us
%improve the final presentation of this article.

\end{document}